\RequirePackage{fix-cm}

\documentclass{article}

\usepackage{fullpage}

\usepackage{amssymb,amsmath}
\usepackage{amsthm}
\usepackage{bm}
\usepackage{color}
\usepackage[numbers,sort]{natbib}

\usepackage{tikz}
\usepackage{pgfplots}
\usetikzlibrary{shapes,arrows,calc,external,matrix,positioning,patterns}

\newcommand{\kuc}{Dolej\v{s}\'{i}-Feistauer-Ku\v{c}era\ }

\newtheorem{theorem}{Theorem}[section]
\newtheorem{corollary}{Corollary}[theorem]
\newtheorem{lemma}[theorem]{Lemma}
\newtheorem{remark}[theorem]{Remark}
\newtheorem{definition}[theorem]{Definition}
\newtheorem{assumption}[theorem]{Assumption}

\DeclareMathOperator{\const}{const}
\DeclareMathOperator{\Id}{Id}

\newcommand{\ww}{\boldsymbol{w}}
\newcommand{\ff}{\boldsymbol{f}}
\newcommand{\nn}{\boldsymbol{n}}

\newcommand{\uu}{\boldsymbol{u}}
\newcommand{\dx}{\,\mathrm{d}x}
\newcommand{\dS}{\,\mathrm{d}\sigma}

\newcommand{\dt}{\Delta t}

\newcommand{\refs}[1]{ #1_{R}}

\newcommand{\R}{\mathbb{R}}
\newcommand{\N}{\mathbb{N}}
\newcommand{\Z}{\mathbb{Z}}
\newcommand{\OO}{\mathcal{O}}

\newcommand{\ffc}{{\boldsymbol f}}

\newcommand{\pw}{\boldsymbol \delta \boldsymbol w}

\newcommand{\HH}{\mathcal H}

\newcommand{\eps}{\varepsilon}
\newcommand{\mm}{\boldsymbol m}

\newcommand{\xx}{\boldsymbol{x}}

\begin{document}

\title{Asymptotic properties of a class of linearly implicit schemes for weakly compressible Euler equations \footnote{The present research has been realized during our Research in Pairs stay at Mathematical Research Institute of Oberwolfach. We want to thank Oberwolfach Institute for this generous support. V.K. was supported by the research project No. 20-01074S of the Czech Science Foundation.
M.L. has been  supported by  the  German Research Foundation (DFG) - Project number 233630050 - TRR 146 as well as by  TRR 165 Waves to Weather. S.N. was funded by DFG grant GRK2326.}}


\author{%
{\sc
V\'{a}clav Ku\v{c}era\thanks{Email: kucera@karlin.mff.cuni.cz}} \\[2pt]
Charles University, Faculty of Mathematics and Physics,\\ Sokolovsk\'{a} 83, Praha 8, 186\,75, Czech Republic.\\[6pt]
{\sc and}\\[6pt]
{\sc M\'{a}ria Luk\'{a}\v{c}ov\'{a}-Medvid'ov\'{a}}\thanks{Email: lukacova@mathematik.uni-mainz.de}\\[2pt]
Universit\"at Mainz, Institut f\"ur Mathematik,\\ Staudingerweg 9, 55128 Mainz, Germany.\\[6pt]
{\sc and}\\[6pt]
{\sc Sebastian Noelle}\thanks{Email: noelle@igpm.rwth-aachen.de}\\[2pt]
Institut f\"ur Geometrie und Praktische Mathematik, RWTH Aachen University,\\ Templergraben 55, 52056 Aachen, Germany.\\[6pt]
{\sc and}\\[6pt]
{\sc Jochen Sch\"{u}tz}\thanks{Email: jochen.schuetz@uhasselt.be}\\[2pt]
Vakgroep wiskunde en statistiek, Universiteit Hasselt,\\ Campus Diepenbeek, Agoralaan Gebouw D, 3590 Diepenbeek, Belgium.
}

\maketitle


\begin{abstract}
In this paper we derive and analyse a class of linearly implicit schemes which includes the one of Feistauer and Ku\v{c}era (JCP 2007) \cite{FeistauerKucera2007} as well as the class of RS-IMEX schemes \cite{NoeSch2014,KSSN16,BispenIMEXSWE,Zakerzadeh2016}. The implicit part is based on a Jacobian matrix which is evaluated at a reference state. This state can be either the solution at the old time level as in \cite{FeistauerKucera2007}, or a numerical approximation of the incompressible limit equations as in \cite{ZeifangSchuetzBeckLukacovaNoelle2019}, or possibly another state.
Subsequently, it is shown that this class of methods is asymptotically preserving under the assumption of a discrete Hilbert expansion. For a one-dimensional setting with some limitations on the reference state, the existence of a discrete Hilbert expansion is shown.


\end{abstract}


\section{Introduction}

We consider multi-dimensional systems of hyperbolic conservation laws that depend on a parameter $\varepsilon \in (0, \varepsilon_0]$, $\eps_0 > 0$ fixed,
\begin{equation}
  \partial_t \ww(\xx,t,\varepsilon)
   + \nabla \cdot \ff(\ww(\xx,t,\varepsilon),\varepsilon)
 =
  0,
\label{SN:conservation_law}
\end{equation}
which are stiff as $\varepsilon$ tends to $0$. Here $(\xx,t) \in \Omega \times \mathbb{R}_+ \subset \mathbb{R}^d \times \mathbb{R}_+ $ are the space-time variables, and
\begin{align}
  \ww: \Omega \times \mathbb{R}_+ \times (0, \varepsilon_0]
 & \to
  \mathcal{N} \subset \mathbb{R}^m
\end{align}
is the solution vector, consisting of the conserved quantities. Here $\mathcal N \subset \R^m$ is a suitable image space depending on the problem at hand, e.g., taking into account positivity of density and the like. The function
\begin{align}
  \ff: \mathcal{N} \times (0, \varepsilon_0]
 & \to
  \mathbb{R}^{m\times d}
\end{align}
is the flux matrix. We assume that for any unit vector
$\nn \in \mathbb{R}^d$ and any $\ww\in\mathcal{N}$, the Jacobian matrix
$\ff'(\ww,\varepsilon)\cdot\nn$ is real diagonalizable with eigenvalues
$\lambda_1(\ww,\varepsilon,\nn), \dots, \lambda_m(\ww,\varepsilon,\nn)$, and that for fixed $\ww$ and $\nn$,
\begin{align}
  \min\limits_{j=1, \dots m}\{|\lambda_j(\ww,\varepsilon,\nn)|
 &=
  \OO(\varepsilon^0)
\\
  \max\limits_{j=1, \dots m}\{|\lambda_j(\ww,\varepsilon,\nn)|
 &=
  \OO(\varepsilon^{-1})
\end{align}
as $\varepsilon\to0$. A classical example  is low Mach number Euler equations of gas dynamics, which is also the system that we will consider in the sequel. A key issue is the choice of time discretization. For explicit schemes, the CFL condition imposes a small time step of order $\OO(\varepsilon \Delta x)$. This might be feasible for a very fast, highly parallel solver such as \cite{MunzDG12} for some given $\eps$, but there exists a threshold on $\eps$ such that for any value  smaller than this threshold, the restriction on $\dt$ becomes too demanding.
Fully implicit schemes, on the other hand, necessitate solving large systems of nonlinear equations, whose condition number deteriorates as the parameter $\varepsilon$ tends to zero. Our focus here is on IMEX (implicit-explicit) schemes \cite{Ascher1995,Ascher1997,pareschi2005implicit,Bo07}, which attempt to split the system into a fast part (treated implicitly) and a slow part (treated explicitly).

Besides the questions of accuracy and efficiency, there is also a qualitative issue of change of type of the system of conservation laws as $\varepsilon$ tends to zero. For instance, weakly compressible solutions
become incompressible in this limit. An important question is whether  this property holds also for the numerical approximation.

The literature on numerical methods for singularly perturbed hyperbolic conservation laws is huge. The interest of this paper is on IMEX schemes for the Euler equations; those schemes necessitate a splitting of the function $\ff(\ww(\xx,t,\varepsilon),\varepsilon)$ into stiff and non-stiff parts. Possible splittings have been introduced in, e.g., 
\cite{Kl95,giraldo2010semi,CoDeKu12,HaJiLi12,BispenIMEXSWE,ArNoLuMu12,ZeifangSchuetzBeckLukacovaNoelle2019}, see also the references in the cited papers.

The linearly implicit scheme presented in \cite{FeistauerKucera2007}, building heavily on the work of \cite{DolejsiFeistauer2004}, is not of the IMEX type; it is presented for the dimensional Euler equations, so there is no (explicit) $\eps-$dependency. Our interest here is to compare asymptotic properties of the scheme \cite{FeistauerKucera2007}, which we call \kuc in this work, with the RS-IMEX scheme presented in \cite{ZeifangSchuetzBeckLukacovaNoelle2019}. Also the latter scheme is a linearly implicit scheme, see  \cite{giraldo2010semi,BispenIMEXSWE,ZeifangSchuetzBeckLukacovaNoelle2019}.

This research has been motivated through the following observation: Although different in type, numerically, both schemes perform very well in the $\eps \rightarrow 0$ limit. For the RS-IMEX scheme, a formal \emph{asymptotic consistency} analysis has been given in \cite{KS17}; no such analysis has been presented for the \kuc scheme.
Even more, the \kuc scheme is not designed to work with the nondimensionalized equations. Nevertheless, consider the convergence results shown in Fig. \ref{fig:numres}. These results show error of a travelling vortex computation for the isentropic Euler equations, see \cite[page 122]{BispenDiss} for details on the vortex. The equations are $\eps-$dependent, and so is the vortex. For $\eps \rightarrow 0$, the equations converge towards the incompressible isentropic Euler equations.
The precise definition of solver parameters are not of importance here, we refer to \cite{FeistauerKucera2007,ZKBSM17} for details. It should be mentioned that they are certainly not comparable (different linear solvers, different triangulations, different numerical fluxes, different representative mesh sizes and so on). The key observation is that both schemes, and not only the RS-IMEX, perform very well for $\eps \rightarrow 0$, which is typically a clear indicator for a scheme being AP.

\begin{figure}[h]
	\begin{center}
        \includegraphics[height=2.in]{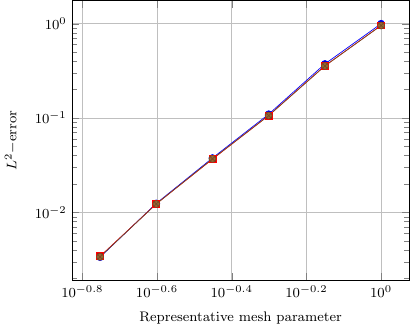}
%
		\includegraphics[height=2.in]{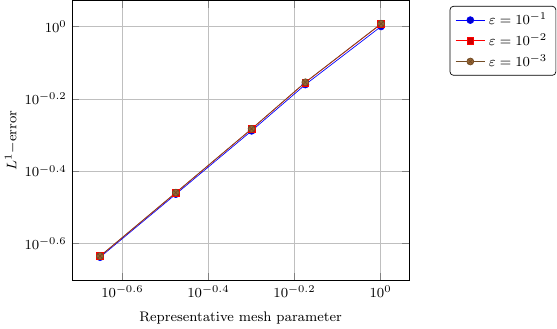}
		\tikzsetnextfilename{results_KUC}
	\caption{Numerical results for the RS-IMEX scheme (left) and the \kuc scheme (right); errors are in density and momentum. It can be seen that the errors are apparently independent of $\eps$, which is typically a good indicator for a scheme being asymptotically consistent. Errors and mesh sizes have been scaled (independently on $\eps$), so that they begin at $(1,1)$. To account for the dimensions in the \kuc scheme, error in density is scaled by $\eps^{-2}$. Please note that quantities on the left and on the right cannot be compared right away.}\label{fig:numres}
	\end{center}
\end{figure}

The contribution of the present paper comes in three parts:
\begin{itemize}
 \item First, we present a unified framework of the RS-IMEX (RS for \emph{reference solution}) and a class of linearly implicit schemes.
\end{itemize}

After having unified the schemes, we solely focus on the full Euler equations of gas dynamics given in form \eqref{SN:conservation_law}, for the ease of presentation formulated in two dimensions, with
\begin{align}
 \label{eq:euler}
 \ww :=
    \begin{pmatrix}
     \rho \\
     \rho u \\
     \rho v \\
     E
    \end{pmatrix},
 \qquad
 \ff :=
    \begin{pmatrix}
        \rho \uu\\
        \rho \uu\otimes\uu+\frac{p}{\varepsilon^2}Id\\
        \uu(E+p)
    \end{pmatrix},
\end{align}
Throughout the paper, $\eps$ denotes a reference Mach number. Here $\uu$ has been defined as the velocity vector $\uu := (u,v)$; the equations come along with the dimensionless equation of state:
\begin{equation}
E=\frac{p}{\gamma-1}+\frac{\varepsilon^2}{2}\rho |\uu|^2.
\label{V_EqState}
\end{equation}
It is known that for $\eps \rightarrow 0$, the solution $\ww$ converges towards the solution of the incompressible equations if initial and boundary data are so-called \emph{well-prepared}, see Def. \ref{V_def_well_posed}, see \cite{schochet1986compressible}; see also \cite{schochet2005mathematical} for a generalization and review of the existing results and \cite{metivier2001incompressible} for a discussion in the case of more generalized initial conditions.

\begin{itemize}
 \item Assuming the existence of an asymptotic expansion of the discretization, we show that the semi-discrete-in-time algorithm converges for $\eps \rightarrow 0$ to a consistent discretization of the incompressible Euler equations. This property has been named in \cite{Jin2012} \emph{asymptotic preserving} (AP), we refer a reader to \cite{AP_1987} where this property has been firstly studied. See also \cite{Guo_Li_Xu} for the so-called unified preserving schemes.

 \item Subsequently, we show under some restrictions that there exists an asymptotic expansion of the semi-discrete-in-time discretization.
\end{itemize}

In this paper, we work with the strong form of the equations. It is hence very important to state the following assumption:
\begin{assumption}
 Throughout the paper, we consider initial data and final times in such a way that $\ww$ remains sufficiently smooth.
\end{assumption}

The paper is organized as follows: In Sec. \ref{section:linearly-implicit-schemes} we introduce the so-called RS-IMEX schemes. We write them as a class of linearly implicit schemes and show that \kuc is a particular, and canonical, member of this class. Sec. \ref{section:asymptotic-consistency} shows that this class of schemes is asymptotically preserving assuming the existence of a discrete Hilbert expansion. Under some restrictions on the reference state, we show in Sec. \ref{section:Hilbert-expansion} that this discrete Hilbert expansion exists in one spatial dimension. Sec. \ref{sec:conout} offers conclusion and outlook.


\section{Linearly implicit schemes based on a reference state }
\label{section:linearly-implicit-schemes}

In this section, we formulate a unified framework containing both the \kuc and the RS-IMEX scheme. For simplicity of exposition, we suppress the dependence on $\eps$ and rewrite (\ref{SN:conservation_law}) as
\begin{equation}
  \partial_t \ww
   + \nabla \cdotp \ff(\ww)
 =
  0.
\end{equation}

\begin{definition}[Flux splitting]
Given a reference state $\ww_R: \Omega \times \R_{+} \rightarrow \mathcal N$, let
\begin{align}
  \widetilde{\ff}(\ww;\ww_R)
 &:=
  \ff(\ww_R) +
   \ff'(\ww_R)(\ww-\ww_R)
\label{ftilde}
\\
  \widehat{\ff}(\ww;\ww_R)
 &:=
  \ff(\ww) -
   \widetilde{\ff}(\ww,\ww_R)
\label{fhat}
\end{align}
be the stiff and non-stiff fluxes.
Note that for fixed $\ww_R$ and $\varepsilon$, the stiff flux $\widetilde{\ff}(\ww;\ww_R)$ is linear in $\ww$.
\end{definition}

The underlying idea is that the Jacobian matrix $\widetilde{\ff}^\prime$ contains all singular eigenvalues (of order $\varepsilon^{-1}$), and $\widetilde \ff$ will hence be discretized implicitly. The Jacobian $\widehat{\ff}^\prime$ contains eigenvalues of order $\varepsilon^{0}$, and $\widehat f$ will hence be discretized explicitly. We call $\widetilde{\ff}$ the stiff and $\widehat{\ff}$ the non-stiff flux.

In the following, we introduce the RS-IMEX scheme, which is based on a reference state that is a function depending on time and space:
\begin{definition}[Time-discretization based on a reference \emph{solution} (RS-IMEX)]
Let $\ww_R^n(\cdot):=\ww_R(\cdot,t^n)$ and $\ww_R^{n+1-}(\cdot):=\ww_R(\cdot,t^{n+1}-0)$. Then the RS-IMEX scheme is given by
\begin{equation}
  \frac{\ww^{n+1} - \ww^{n}}{\Delta t}
 =
  - \nabla\cdotp
   \left( \widetilde{\ff}(\ww^{n+1};\ww_R^{n+1-})
    + \widehat{\ff}(\ww^{n};\ww_R^n) \right).
\label{RS-IMEX}
\end{equation}
\end{definition}

Next we introduce a variant of the RS-IMEX scheme which is not based on a reference solution $\ww_R(t)$, but on a reference state $\overline{\ww}_R^n$ which is constant in the time interval $[t^n,t^{n+1})$ (but possibly variable in space):
\begin{definition}[IMEX time-discretization based on a reference \emph{state}]
\label{SN:def:RS-IMEX_refstate} Here we suppose that
$\ww_R(t) \equiv \overline{\ww}_R^n$
is constant in time on the interval $[t^n,t^{n+1})$. We call $\overline{\ww}_R^n: \Omega \rightarrow \mathcal N$ the reference state. Then the RS-IMEX scheme based on a reference state is given by
\begin{equation}
  \frac{\ww^{n+1} - \ww^{n}}{\Delta t}
 =
  - \nabla\cdotp
   \left( \widetilde{\ff}(\ww^{n+1};\overline{\ww}_R^n)
    + \widehat{\ff}(\ww^{n};\overline{\ww}_R^n) \right).
\label{RState-IMEX}
\end{equation}
\end{definition}
The following lemma considerably simplifies the form of the scheme (\ref{RState-IMEX}). It also provides a convenient basis for a DG space discretization:

\begin{lemma}[Linearly implicit scheme based on a reference state]
The scheme (\ref{RState-IMEX}) is equivalent to the linearly implicit scheme
\begin{align}
  \frac{\ww^{n+1} - \ww^{n}}{\Delta t}
 &=
  - \nabla\cdotp \big( \ff(\ww^n) + \ff^\prime(\overline{\ww}_R^n)
   ( \ww^{n+1}-\ww^n)  \big).
\label{V_Unified_scheme}
\end{align}
\end{lemma}

\begin{proof}
From (\ref{ftilde}) and (\ref{fhat}),
\begin{align}
  \widetilde{\ff}&(\ww^{n+1};\overline{\ww}_R^n)
   - \widehat{\ff}(\ww^{n};\overline{\ww}_R^n)
\nonumber\\
 &=
  \Big(
   \ff(\overline{\ww}_R^n)
    + \ff'(\overline{\ww}_R^n)
     (\ww^{n+1}-\overline{\ww}_R^n)  \Big)
  + \left(
   \ff(\ww^n) -
    \widetilde{\ff}(\ww^n,\overline{\ww}_R^n)
\right)
\nonumber\\
 &=
  \Big(
   \ff(\overline{\ww}_R^n)
    + \ff'(\overline{\ww}_R^n)
     (\ww^{n+1}-\overline{\ww}_R^n)  \Big)
  +
   \ff(\ww^n) -
    \Big( \ff(\overline{\ww}_R^n)
     + \ff'(\overline{\ww}_R^n)(\ww^{n}-\overline{\ww}_R^n)
     \Big)
\nonumber\\
 &=
  \Big(
    \ff'(\overline{\ww}_R^n)
     (\ww^{n+1}-\overline{\ww}_R^n)  \Big)
  +
   \ff(\ww^n) -
    \Big(
    \ff'(\overline{\ww}_R^n)(\ww^{n}-\overline{\ww}_R^n)
     \Big)
\nonumber\\
 &=
  \ff(\ww^n) +
   \ff'(\overline{\ww}_R^n)(\ww^{n+1}-\ww^n).
\nonumber
\end{align}
\end{proof}


\begin{remark}
Taking the reference state to be the discretization at time level $n$, i.e., $\overline{\ww}_R^n = \ww^n$, then (\ref{V_Unified_scheme}) reduces to the classical linear implicit scheme
\begin{align}
  \frac{\ww^{n+1} - \ww^{n}}{\Delta t}
 &=
  - \nabla\cdotp \Big( \ff(\ww^n) + \ff^\prime(\ww^n) ( \ww^{n+1}-\ww^n)  \Big)
\label{classical-linearly-implicit}
\end{align}
If, in addition, the flux is homogeneous of degree one, i.e. $\ff(\ww) = \ff^\prime(\ww) \ww$, then
\begin{align}
  \frac{\ww^{n+1} - \ww^{n}}{\Delta t}
 &=
  - \nabla\cdotp \Big( \ff^\prime(\ww^n) \ww^{n+1} \Big).
\label{classical-homogeneous-linearly-implicit}
\end{align}
This is at the basis of the \kuc scheme, proposed in \cite{DolejsiFeistauer2004,FeistauerKucera2007} for the Euler equations of gas dynamics.
\end{remark}
\begin{remark}
In his dissertation \cite{Kaiser2018-Diss}, Kaiser observed that for the full Euler equations in multiple space dimensions, the Jacobian of the non-stiff flux, $\widehat{\ff}^\prime(\overline{\ww}_R^n)$, may have complex eigenvalues if the tangential velocities are large enough compared with the normal velocities. This was remedied in \cite{ZeifangSchuetzBeckLukacovaNoelle2019} by removing terms of order $\varepsilon^2$ from the linearized equation of state.
\end{remark}

\begin{remark}
 To simplify the notation, we will usually omit the bar at $\ww_R$ and write simply
\begin{align}
\ww_R =: (\rho_R, {\rho_R} u_R, {\rho_R} v_R, E_R)
\end{align}
whenever this does not lead to confusion.
From now on, it is assumed that $\ww_R$ is of the form given in Def. \ref{SN:def:RS-IMEX_refstate}.
\end{remark}

In Section~\ref{section:asymptotic-consistency}, we study the asymptotic consistency of the RS-IMEX scheme given in Definition~\ref{SN:def:RS-IMEX_refstate} for the two-dimensional Euler equations of gas dynamics. In Section~\ref{section:Hilbert-expansion}, we specialize to the one-dimensional case and a constant reference solution $\ww_R$ and prove the existence of an asymptotic expansion for our class of linearly implicit schemes.

\bigskip

\section{AP analysis}
\label{section:asymptotic-consistency}

Considering the Euler fluxes \eqref{eq:euler} and defining $\ww:=(w_1,w_2,w_3,w_4)^T$, one can write the two Euler fluxes in terms of $\ww$ as
\begin{equation}
\ff_1(\ww)=\begin{pmatrix}
w_2\\
\frac{3-\gamma}{2}\frac{w_2^2}{w_1} +\frac{1-\gamma}{2}\frac{w_3^2}{w_1} +\frac{\gamma-1}{\varepsilon^2}w_4\\
\frac{w_2w_3}{w_1}\\
\frac{\gamma w_2 w_4}{w_1} -\frac{\varepsilon^2(\gamma-1)}{2}\frac{w_2^3+w_2w_3^2}{w_1^2}
\end{pmatrix},\quad
\ff_2(\ww)=\begin{pmatrix}
w_3\\
\frac{w_2w_3}{w_1}\\
\frac{1-\gamma}{2}\frac{w_2^2}{w_1} +\frac{3-\gamma}{2}\frac{w_3^2}{w_1} +\frac{\gamma-1}{\varepsilon^2}w_4\\
\frac{\gamma w_3 w_4}{w_1} -\frac{\varepsilon^2(\gamma-1)}{2}\frac{w_2^2w_3+w_3^3}{w_1^2}
\end{pmatrix}.
\end{equation}
Using this notation \eqref{SN:conservation_law} reads
\begin{equation}
\partial_t\ww+\partial_x\ff_1(\ww)+\partial_y\ff_2(\ww)=0.
\end{equation}
Jacobi matrices of $\ff_1$ and $\ff_2$ with respect to $\ww$ (written in terms of the physical variables density $\rho$, momentum $\rho \uu$ and energy $E$) are given by
\begin{align}
\ff^\prime_1(\ww)=
\begin{pmatrix}
0&1&0&0\\
\frac{\gamma-3}{2}u^2+\frac{\gamma-1}{2}v^2 &(3-\gamma)u&(1-\gamma)v&\frac{\gamma-1}{\varepsilon^2}\\
-uv&v&u&0\\
-\frac{\gamma Eu}{\rho}+\varepsilon^2(\gamma-1)u(u^2+v^2),& \frac{\gamma E}{\rho}-\varepsilon^2\frac{\gamma-1}{2}(3u^2+v^2), &\varepsilon^2(1-\gamma)uv,&\gamma u
\end{pmatrix},
\label{V_Jacobians1}\\
\ff^\prime_2(\ww)=
\begin{pmatrix}
0&0&1&0\\
-uv&v&u&0\\
\frac{\gamma-1}{2}u^2+\frac{\gamma-3}{2}v^2 &(1-\gamma)u&(3-\gamma)v&\frac{\gamma-1}{\varepsilon^2}\\
-\frac{\gamma Ev}{\rho}+\varepsilon^2(\gamma-1)v(u^2+v^2),& \varepsilon^2(1-\gamma)uv,& \frac{\gamma E}{\rho}-\varepsilon^2\frac{\gamma-1}{2}(u^2+3v^2), &\gamma v
\end{pmatrix}.
\label{V_Jacobians2}
\end{align}

We fix the boundary conditions as follows:
\begin{assumption}
In the following we assume either periodic boundary conditions or slip (wall) boundary conditions for the velocity: $\uu\cdotp\nn=0$ on $\partial\Omega$, where $\nn$ is the unit outer normal to $\Omega$.
\end{assumption}



\subsection{Formal expansion of the scheme}
We make the following formal assumption on the existence of a Hilbert expansion. For the validity of this assumption, we refer the reader to Sec. \ref{section:Hilbert-expansion}.
\begin{assumption}
We assume that the physical quantities $\rho,\uu,E$ and $p$ on each time level have a formal Hilbert expansion of the form (written e.g. for $\rho^n$)
\begin{equation}
\rho^n(x)=\rho^n_{(0)}(x) +\varepsilon\rho^n_{(1)}(x) +\varepsilon^2\rho^n_{(2)}(x)+\OO(\varepsilon^3),
\label{V_rho_expansion}
\end{equation}
similarly, this is assumed for the reference state $\ww_R$.
\end{assumption}

\begin{remark}
 It is trivial that $\ww_R$ used in the RS-IMEX \cite{ZeifangSchuetzBeckLukacovaNoelle2019} has a Hilbert expansion, because it does not depend on $\eps$. For the \kuc scheme \cite{FeistauerKucera2007}, however, this is not clear, as $\ww_R$ is the solution from the previous time iterate.
\end{remark}
Substituting the Hilbert expansions into the expressions (\ref{V_Jacobians1}) and (\ref{V_Jacobians2}) gives the expansion
\begin{equation}
\ff^\prime_s(\ww)=\varepsilon^{-2}\ff^\prime_{s,(-2)}(\ww) +\varepsilon^{-1}\ff^\prime_{s,(-1)}(\ww) +\varepsilon^{0}\ff^\prime_{s,(0)}(\ww) +\OO(\varepsilon),
\label{V_Fprime}
\end{equation}
for $s = 1, 2$, where
\begin{equation}
\ff^\prime_{1,(-2)}(\ww)=
\begin{pmatrix}
0&0&0&0\\
0&0&0&(\gamma-1)\\
0&0&0&0\\
0&0&0&0
\end{pmatrix},\quad
\ff^\prime_{2,(-2)}(\ww)=
\begin{pmatrix}
0&0&0&0\\
0&0&0&0\\
0&0&0&(\gamma-1)\\
0&0&0&0
\end{pmatrix}
\end{equation}
and $\ff^\prime_{s,(-1)}(\ww)=0$ for $s=1,2$. Finally, since
\begin{equation}
\frac{1}{\rho}=\frac{1}{\rho_{(0)}} -\frac{\rho_{(1)}}{\rho_{(0)}^2}\varepsilon +\OO(\varepsilon^2)
\label{V_rho_Taylor}
\end{equation}
due to the Taylor expansion, we have
\begin{align}
\ff^\prime_{1,(0)}(\ww)=
\begin{pmatrix}
0&1&0&0\\
\frac{\gamma-3}{2}u_{(0)}^2+\frac{\gamma-1}{2}v_{(0)}^2, &(3-\gamma)u_{(0)},&(1-\gamma)v_{(0)},&0\\
-u_{(0)}v_{(0)}&v_{(0)}&u_{(0)}&0\\
-\frac{\gamma E_{(0)}u_{(0)}}{\rho_{(0)}}& \frac{\gamma E_{(0)}}{\rho_{(0)}} &0&\gamma u_{(0)}
\end{pmatrix},
\label{V_Jacobians1_0}
\\
\ff^\prime_{2,(0)}(\ww)=
\begin{pmatrix}
0&0&1&0\\
-u_{(0)}v_{(0)}&v_{(0)}&u_{(0)}&0\\
\frac{\gamma-1}{2}u_{(0)}^2+\frac{\gamma-3}{2}v_{(0)}^2, &(1-\gamma)u_{(0)},&(3-\gamma)v_{(0)},&0\\
-\frac{\gamma E_{(0)}v_{(0)}}{\rho_{(0)}}& 0& \frac{\gamma E_{(0)}}{\rho_{(0)}} &\gamma v_{(0)}
\end{pmatrix}.
\label{V_Jacobians2_0}
\end{align}

Taking all the expansions (\ref{V_rho_expansion}) -- (\ref{V_Jacobians2_0}) and substituting into the linearized problem (\ref{V_Unified_scheme}), we gather terms according to the powers of $\varepsilon$. For $\varepsilon^{-2}$ and $\varepsilon^{-1}$ we get the following lemma.

\begin{lemma}
\label{V_Lemma1}
The functions $E^{n}_{(0)}, E^{n}_{(1)}, p^{n}_{(0)}$ and $p^{n}_{(1)}$ are constant in space for every $n$.
\end{lemma}
\proof
By gathering the terms of order $\varepsilon^{-2}$ and $\varepsilon^{-1}$ from (\ref{V_Unified_scheme}), we obtain
\begin{align}
\nabla\big(p_{(0)}^n +(\gamma-1)(E_{(0)}^{n+1}-E_{(0)}^n)\big)&=0,
\label{V_minustwo-order-eq1a}
\\
\nabla\big(p_{(1)}^n +(\gamma-1)(E_{(1)}^{n+1}-E_{(1)}^n)\big)&=0.
\label{V_minustwo-order-eq1b}
\end{align}
Taking the $\varepsilon^{0}$ and $\varepsilon^{1}$ terms from the equation of state (\ref{V_EqState}) at time level $n$ gives
\begin{equation}
E_{(0)}^n=\frac{p_{(0)}^n}{\gamma-1},\quad
E_{(1)}^n=\frac{p_{(1)}^n}{\gamma-1}.
\label{V_minustwo-order-eq2}
\end{equation}
Substituting into (\ref{V_minustwo-order-eq1a}) and (\ref{V_minustwo-order-eq1b}) gives $\nabla E_{(0)}^{n+1}=\nabla E_{(1)}^{n+1}=0$, hence $E_{(0)}^{n+1}$ and $E_{(1)}^{n+1}$ are constant in space for every $n$. Equation (\ref{V_minustwo-order-eq2}) implies the same for $p_{(0)}^{n+1}$ and $p_{(1)}^{n+1}$.

\qed

Collecting the $\varepsilon^{0}$ terms of the mass equation from (\ref{V_Unified_scheme}) gives
\begin{equation}
\frac{\rho^{n+1}_{(0)}-\rho^{n}_{(0)}}{\Delta t} +\nabla\cdotp(\rho^{n+1}_{(0)}\uu^{n+1}_{(0)})=0, \label{V_zero-order-eq1}
\end{equation}
Similarly, from the momentum equation we get
\begin{align}
&\frac{\rho^{n+1}_{(0)}u^{n+1}_{(0)}-\rho^{n}_{(0)}u^{n}_{(0)}}{\Delta t} +\partial_x\bigg(\rho_{(0)}^n(u_{(0)}^n)^2+p_{(2)}^n \nonumber\\
&\quad+\big(\tfrac{\gamma-3}{2}u_{R,(0)}^2 +\tfrac{\gamma-1}{2}v_{R,(0)}^2\big)(\rho_{(0)}^{n+1}-\rho_{(0)}^n)
+(3-\gamma)u_{R,(0)}(\rho_{(0)}^{n+1}u_{(0)}^{n+1}-\rho_{(0)}^nu_{(0)}^{n})
\nonumber\\
&\quad+(1-\gamma)v_{R,(0)}(\rho_{(0)}^{n+1}v_{(0)}^{n+1}-\rho_{(0)}^nv_{(0)}^{n}) +(\gamma-1)(E_{(2)}^{n+1}-E_{(2)}^n)\bigg)\nonumber\\
&\quad+\partial_y\bigg(\rho_{(0)}^n u_{(0)}^n v_{(0)}^n -u_{R,(0)}v_{R,(0)}(\rho_{(0)}^{n+1}-\rho_{(0)}^n) +v_{R,(0)}(\rho_{(0)}^{n+1}u_{(0)}^{n+1}-\rho_{(0)}^nu_{(0)}^{n})
\nonumber\\
&\quad+u_{R,(0)}(\rho_{(0)}^{n+1}v_{(0)}^{n+1}-\rho_{(0)}^nv_{(0)}^{n})
\bigg)=0
\label{V_zero-order-eq2}
\end{align}
and
\begin{align}
&\frac{\rho^{n+1}_{(0)}v^{n+1}_{(0)}-\rho^{n}_{(0)}v^{n}_{(0)}}{\Delta t} +\partial_x\bigg(\rho_{(0)}^n u_{(0)}^n v_{(0)}^n -u_{R,(0)}v_{R,(0)}(\rho_{(0)}^{n+1}-\rho_{(0)}^n)\nonumber \\ &\quad+v_{R,(0)}(\rho_{(0)}^{n+1}u_{(0)}^{n+1}-\rho_{(0)}^nu_{(0)}^{n})
+u_{R,(0)}(\rho_{(0)}^{n+1}v_{(0)}^{n+1}-\rho_{(0)}^nv_{(0)}^{n})
\bigg)\nonumber \\ &\quad+\partial_y\bigg(\rho_{(0)}^n(v_{(0)}^n)^2+p_{(2)}^n +\big(\tfrac{\gamma-1}{2}u_{R,(0)}^2 +\tfrac{\gamma-3}{2}v_{R,(0)}^2\big)(\rho_{(0)}^{n+1}-\rho_{(0)}^n)
\nonumber\\
&\quad+(1-\gamma)u_{R,(0)}(\rho_{(0)}^{n+1}u_{(0)}^{n+1}-\rho_{(0)}^nu_{(0)}^{n})
+(3-\gamma)v_{R,(0)}(\rho_{(0)}^{n+1}v_{(0)}^{n+1}-\rho_{(0)}^nv_{(0)}^{n})\nonumber \\ 
&\quad+(\gamma-1)(E_{(2)}^{n+1}-E_{(2)}^n)
\bigg)=0.
\label{V_zero-order-eq3}
\end{align}
Finally from the energy equation we get
\begin{align}
&\frac{E^{n+1}_{(0)}-E^{n}_{(0)}}{\Delta t} +\nabla\cdotp\bigg( \big(E^n_{(0)}+p^n_{(0)}\big)\uu^n_{(0)} -\gamma\frac{E_{R,(0)}\uu_{R,(0)}}{\rho_{R,(0)}} (\rho^{n+1}_{(0)}-\rho^{n}_{(0)})\nonumber \\ &\quad+\gamma\frac{E_{R,(0)}}{\rho_{R,(0)}} (\rho^{n+1}_{(0)}\uu^{n+1}_{(0)}-\rho^{n}_{(0)}\uu^{n}_{(0)})
+\gamma \uu_{R,(0)}(E^{n+1}_{(0)}-E^{n}_{(0)})\bigg)=0. \label{V_zero-order-eq4}
\end{align}

We note that if we assume periodic or slip boundary conditions e.g. for $\uu^n$, then the same boundary conditions hold for the individual terms in its Hilbert expansion. This can be seen (e.g. in the case of slip boundary conditions) by taking the limit $\varepsilon\to 0$ in the boundary condition $\uu^n\cdotp\nn=0$, which immediately gives $\uu^n_{(0)}\cdotp\nn=0$. Then we have $0=\varepsilon\uu^n_{(1)}\cdotp\nn +\varepsilon^2\uu^n_{(2)}\cdotp\nn+\OO(\varepsilon^3)$ which we can divide {by} $\varepsilon$ and take $\varepsilon\to 0$ to obtain  $\uu^n_{(1)}\cdotp\nn=0$. Similarly  $\uu^n_{(2)}\cdotp\nn=0$, etc.

\begin{lemma}
\label{V_Lemma1a}
Assuming either slip boundary conditions for $\uu_R$ and $\uu^n$ for all $n$ or periodic boundary conditions, the functions $E^{n}_{(0)}$ and $p^{n}_{(0)}$ are constant in space and independent of $n$.
\end{lemma}
\proof
We integrate (\ref{V_zero-order-eq4}) over $\Omega$ and apply Green's theorem. Since $E^{n}_{(0)}$ and $E^{n+1}_{(0)}$ are constant by Lemma \ref{V_Lemma1}, we get
\begin{equation}
|\Omega|\frac{E^{n+1}_{(0)}-E^{n}_{(0)}}{\Delta t} +\int_{\partial\Omega}\mathcal{E}\cdotp\nn\dS=0,
\label{V_zero-order-eq3a}
\end{equation}
where $\mathcal{E}$ corresponds to the terms under the divergence symbol in (\ref{V_zero-order-eq4}). Since each of these terms contains either $\uu_{R,(0)}, \uu^n_{(0)}$ or $\uu^{n+1}_{(0)}$, all of which have zero normal component on $\partial\Omega$, the whole boundary integral in (\ref{V_zero-order-eq3a}) vanishes. This is the case of slip-boundary conditions, for periodic boundary conditions, the boundary integral vanishes due to spatial periodicity of all the terms. Altogether, (\ref{V_zero-order-eq3a}) then implies $E^{n+1}_{(0)}=E^{n}_{(0)}$ and (\ref{V_minustwo-order-eq2}) implies $p^{n+1}_{(0)}=p^{n}_{(0)}$.

\qed


\subsection{Asymptotic preserving property}
In this section we prove that the zero order variables from the Hilbert expansion satisfy the incompressible Euler equations. First, we start with the incompressibility.

\begin{lemma}
\label{V_Lemma2}
Assume either slip boundary conditions for $\uu_R$ and $\uu^n$ for all $n$ or periodic boundary conditions. Let $\rho^{n}_{(0)}$ and $\rho_{R,(0)}$ be constant in space and let $\nabla\cdotp \uu^{n}_{(0)} =\nabla \cdotp \uu_{R,(0)}=0$. Then $\rho^{n+1}_{(0)}=\rho^{n}_{(0)}$, i.e. $\rho^{n+1}_{(0)}$ is also constant in space, and $\nabla\cdotp \uu^{n+1}_{(0)}=0$.
\end{lemma}
\proof
We can simplify the energy equation (\ref{V_zero-order-eq4}) using Lemma \ref{V_Lemma1a} and the assumptions  $\nabla\rho^{n}_{(0)}=0$ and $\nabla\cdotp \uu^{n}_{(0)} =\nabla\cdotp \uu_{R,(0)}=0$ to obtain
\begin{equation}
-\uu_{R,{(0)}}\cdotp\nabla(\rho^{n+1}_{(0)}-\rho^{n}_{(0)}) +\nabla\cdotp(\rho^{n+1}_{(0)}\uu^{n+1}_{(0)})=0.
\label{V_Lemma2_eq1}
\end{equation}
Substituting this equality into the mass equation (\ref{V_zero-order-eq1}) gives us
\begin{equation}
\frac{\rho^{n+1}_{(0)}-\rho^{n}_{(0)}}{\Delta t} + \uu_{R,(0)}\cdotp\nabla\big(\rho^{n+1}_{(0)}-\rho^{n}_{(0)}\big)=0.
\label{V_Lemma2_eq2}
\end{equation}
Denoting for simplicity $\varrho:=\rho^{n+1}_{(0)}-\rho^{n}_{(0)}$, we write (\ref{V_Lemma2_eq2}) as
\begin{equation}
\tfrac{1}{\Delta t}\varrho +\uu_{R,(0)}\cdotp\nabla\varrho=0.
\label{V_Lemma2_eq3}
\end{equation}
We wish to prove that $\varrho=0$, i.e., that $\rho^{n+1}_{(0)}=\rho^{n}_{(0)}$. To this end, we multiply (\ref{V_Lemma2_eq3}) by $\varrho$ and integrate over $\Omega$:
\begin{equation}
\frac{1}{\Delta t}\int_\Omega\varrho^2\dx +\int_\Omega\uu_{R,(0)}\cdotp\nabla\varrho\,\varrho\dx=0.
\label{V_Lemma2_eq4}
\end{equation}
We apply Green's theorem to the second integral to obtain
\begin{equation}
\int_\Omega\uu_{R,(0)}\cdotp\nabla\varrho\,\varrho\dx =\underbrace{\int_{\partial\Omega} \uu_{R,(0)}\cdotp\nn\varrho^2\dS}_{=0} -\underbrace{\int_\Omega\nabla\cdotp \uu_{R,(0)}\varrho^2\dx}_{=0} -\int_\Omega \uu_{R,(0)}\cdotp\nabla\varrho\,\varrho\dx,
\label{V_Lemma2_eq5}
\end{equation}
where the first and second right-hand side terms are zero due to the boundary conditions and the divergence-free assumption on $ \uu_{R,(0)}$, respectively, while the last term equals the left-hand side. Therefore, (\ref{V_Lemma2_eq5}) gives us $\int_\Omega \uu_{R,(0)}\cdotp\nabla\varrho\,\varrho\dx=0$, which together with (\ref{V_Lemma2_eq4}) implies
\begin{equation}
\frac{1}{\Delta t}\int_\Omega\varrho^2\dx=0\quad\Longrightarrow\quad \varrho=0 \text{ a.e. in } \Omega\quad\Longrightarrow\quad \rho^{n+1}_{(0)}=\rho^{n}_{(0)}.
\label{V_Lemma2_eq6}
\end{equation}
Thus we have obtained the first statement of the Lemma.

Finally, since we now know that $\nabla \rho^{n+1}_{(0)}=\nabla \rho^{n}_{(0)}=0$, equation (\ref{V_Lemma2_eq1}) simplifies to $\nabla\cdotp \uu^{n+1}_{(0)} =0$, which completes the proof.

\qed

Now we prove that the lowest order terms in the Hilbert expansion satisfy the semi-discrete incompressible Euler equations, implicitly discretized in time. One then has an $\OO(\Delta t)$ consistency error which comes from the time discretization and a consistency error arising due to the linearization of the fluxes. As we shall mention later, for the \kuc and RS-IMEX schemes this consistency error is of the order $\OO(\Delta t^2)$.

\begin{theorem}
\label{V_Theorem1}
Let the initial condition satisfy $\nabla\cdotp \uu^{0}_{(0)}=0$ and $\rho^{0}_{(0)}$ being constant in space. Let the reference solution satisfy $\nabla\cdotp \uu_{R,(0)}^n=0$ and $\rho^{n}_{R,(0)}$ being constant in space for all $n$. Assume either slip boundary conditions for $\uu_R^n$ and $\uu^n$ for all $n$ or periodic boundary conditions. Then for each $n$, the pair $\left(\uu^{n+1}_{(0)},p^{n+1}_{(2)}/\rho_{(0)}^{n+1}\right)$ solves the implicit semi-discrete incompressible Euler equations
\begin{equation}
\begin{split}
\frac{\uu^{n+1}_{(0)}-\uu^{n}_{(0)}}{\Delta t} +\nabla\cdotp\Big(\uu_{(0)}^{n+1}\otimes \uu_{(0)}^{n+1}\Big) +\nabla\frac{p_{(2)}^{n+1}}{\rho_{(0)}^{n+1}} &=\mathcal{E}^{n+1},\\
\nabla\cdotp \uu^{n+1}_{(0)}&=0,
\label{V_Theorem1_1}
\end{split}
\end{equation}
where $\mathcal{E}^{n+1}$ is a consistency error term satisfying
\begin{equation}
|\mathcal{E}^{n+1}|\leq C\|\uu^{n+1}_{(0)}-\uu^{n}_{(0)}\|_{W^{1,\infty}} \Big(\|\uu^{n+1}_{(0)}-\uu^{n}_{(0)}\|_{W^{1,\infty}} +\|\uu^{n}_{(0)}-\uu^{n}_{R,(0)}\|_{W^{1,\infty}}\Big),
\label{V_Theorem1_2}
\end{equation}
where $C$ depends only on $\gamma$.
\end{theorem}

\proof 
Lemma \ref{V_Lemma2} implies that $u_{(0)}^{n+1}$ is divergence-free. To show the first part of (\ref{V_Theorem1_1}), we will work with equation (\ref{V_zero-order-eq2}) for the $x$-component of momentum, equation (\ref{V_zero-order-eq3}) can be treated similarly. Since $\rho_{(0)}^n=\rho_{(0)}^{n+1}$ is constant in space due to Lemma \ref{V_Lemma2}, we can divide (\ref{V_zero-order-eq2}) by density and simplify:
\begin{align}
&\frac{u^{n+1}_{(0)}-u^{n}_{(0)}}{\Delta t} +\partial_x\bigg((u_{(0)}^n)^2+\frac{p_{(2)}^n}{\rho_{(0)}^n} +(3-\gamma)u_{R,(0)}(u_{(0)}^{n+1}-u_{(0)}^{n})\nonumber\\
&\quad+(1-\gamma)v_{R,(0)}(v_{(0)}^{n+1}-v_{(0)}^{n}) +\frac{\gamma-1}{\rho_{(0)}^n}(E_{(2)}^{n+1}-E_{(2)}^n)\bigg)\nonumber\\
&\quad+\partial_y\Big(u_{(0)}^n v_{(0)}^n +v_{R,(0)}(u_{(0)}^{n+1}-u_{(0)}^{n})
+u_{R,(0)}(v_{(0)}^{n+1}-v_{(0)}^{n})
\Big)=0.
\label{V_Theorem1_eq1}
\end{align}
The pressure and energy terms from (\ref{V_Theorem1_eq1}) can be expressed using the equation of state (\ref{V_EqState}), namely by considering its $\OO(\varepsilon^2)$ terms
\begin{align}
 E_{(2)} = \frac{p_{(2)}}{\gamma - 1} + \frac1 2 {\rho_{(0)} |\uu_{(0)}|^2}.
\end{align}
We obtain
\begin{align}
&\frac{1}{\rho_{(0)}^n}\bigg(p_{(2)}^n +(\gamma-1)(E_{(2)}^{n+1}-E_{(2)}^n)\bigg) \nonumber\\
&\quad=\frac{1}{\rho_{(0)}^n}\bigg(p_{(2)}^n +(\gamma-1)\bigg(\frac{p_{(2)}^{n+1}}{\gamma-1} +\frac{1}{2}\rho_{(0)}^{n+1}|\uu_{(0)}^{n+1}|^2 -\frac{p_{(2)}^{n}}{\gamma-1} -\frac{1}{2}\rho_{(0)}^{n}|\uu_{(0)}^{n}|^2 \bigg)\bigg)\nonumber\\
&\quad=\frac{p_{(2)}^{n+1}}{\rho_{(0)}^{n+1}} +\frac{\gamma-1}{2}\Big(|\uu_{(0)}^{n+1}|^2 -|\uu_{(0)}^{n}|^2 \Big).
\label{V_Theorem1_eq2}
\end{align}
Substituting (\ref{V_Theorem1_eq2}) into (\ref{V_Theorem1_eq1}) leads to
\begin{align}
&\frac{u^{n+1}_{(0)}-u^{n}_{(0)}}{\Delta t} +\partial_x\bigg((u_{(0)}^n)^2 +(3-\gamma)u_{R,(0)}(u_{(0)}^{n+1}-u_{(0)}^{n})\nonumber\\
&\quad+(1-\gamma)v_{R,(0)}(v_{(0)}^{n+1}-v_{(0)}^{n}) +\frac{p_{(2)}^{n+1}}{\rho_{(0)}^{n+1}} +\frac{\gamma-1}{2}\Big(|\uu_{(0)}^{n+1}|^2 -|\uu_{(0)}^{n}|^2 \Big)\bigg)\nonumber\\
&\quad +\partial_y\Big(u_{(0)}^n v_{(0)}^n +v_{R,(0)}(u_{(0)}^{n+1}-u_{(0)}^{n})
+u_{R,(0)}(v_{(0)}^{n+1}-v_{(0)}^{n})
\Big)=0.
\label{V_Theorem1_eq3}
\end{align}
We now collect all the terms under the $\partial_x$ symbol in (\ref{V_Theorem1_eq3}) which contain the $x$-component of $\uu$ or $\uu_R$:
\begin{align}
(u_{(0)}^n&)^2 +(3-\gamma)u_{R,(0)}(u_{(0)}^{n+1}-u_{(0)}^{n})
+\tfrac{\gamma-1}{2}\big((u_{(0)}^{n+1})^2 -(u_{(0)}^{n})^2 \big)\nonumber\\
&=(u_{(0)}^{n+1})^2 -(u_{(0)}^{n+1})^2 +(u_{(0)}^{n})^2 +(3-\gamma)u_{R,(0)}(u_{(0)}^{n+1}-u_{(0)}^{n})
+\tfrac{\gamma-1}{2}\big((u_{(0)}^{n+1})^2 -(u_{(0)}^{n})^2 \big)\nonumber\\
&=(u_{(0)}^{n+1})^2 +\tfrac{\gamma-3}{2}\big(u_{(0)}^{n+1} -u_{(0)}^{n}\big)\big(u_{(0)}^{n+1} +u_{(0)}^{n}-2u_{R,(0)}\big).
\label{V_Theorem1_eq4}
\end{align}
Similarly, we collect all the terms under the $\partial_x$ symbol in (\ref{V_Theorem1_eq3}) which contain the $y$-component of $\uu$ or $\uu_R$:
\begin{align}
(1-\gamma)v_{R,(0)}&(v_{(0)}^{n+1}-v_{(0)}^{n}) +\tfrac{\gamma-1}{2}\big((v_{(0)}^{n+1})^2 -(v_{(0)}^{n})^2 \big)=\tfrac{\gamma-1}{2}(v_{(0)}^{n+1}-v_{(0)}^{n}) \big(v_{(0)}^{n+1} +v_{(0)}^{n} -2v_{R,(0)} \big).
\label{V_Theorem1_eq5}
\end{align}
Now we take all the terms under the $\partial_y$ symbol in (\ref{V_Theorem1_eq3}):
\begin{align}
u_{(0)}^n& v_{(0)}^n +v_{R,(0)}(u_{(0)}^{n+1}-u_{(0)}^{n})
+u_{R,(0)}(v_{(0)}^{n+1}-v_{(0)}^{n})\nonumber\\
&=u_{(0)}^{n+1}v_{(0)}^{n+1} -u_{(0)}^{n+1}v_{(0)}^{n+1}+u_{(0)}^{n}v_{(0)}^{n} +v_{R,(0)}(u_{(0)}^{n+1}-u_{(0)}^{n})+u_{R,(0)}(v_{(0)}^{n+1}-v_{(0)}^{n})
\nonumber\\
&=u_{(0)}^{n+1}v_{(0)}^{n+1}
-(v_{(0)}^{n+1}-v_{(0)}^n)(u_{(0)}^{n+1}-u_{R,(0)}) -(u_{(0)}^{n+1}-u_{(0)}^n)(v_{(0)}^n-v_{R,(0)}).
\label{V_Theorem1_eq6}
\end{align}
Altogether, if we substitute (\ref{V_Theorem1_eq4})--(\ref{V_Theorem1_eq6}) into the momentum equation (\ref{V_Theorem1_eq3}) we get
\begin{equation}
\frac{u^{n+1}_{(0)}-u^{n}_{(0)}}{\Delta t} +\partial_x\bigg((u_{(0)}^{n+1})^2  +\frac{p_{(2)}^{n+1}}{\rho_{(0)}^{n+1}}\bigg) +\partial_y\Big(u_{(0)}^{n+1}v_{(0)}^{n+1}\Big)
=E_1+E_2,
\label{V_Theorem1_eq7}
\end{equation}
This equation is simply the backward Euler discretization of the equation for the $x$-component of velocity from the incompressible Euler equations with error terms
\begin{align}
E_1&=-\partial_x\Big(\tfrac{\gamma-3}{2}\big(u_{(0)}^{n+1} -u_{(0)}^{n}\big)\big(u_{(0)}^{n+1} +u_{(0)}^{n}-2u_{R,(0)}\big) +\tfrac{\gamma-1}{2}(v_{(0)}^{n+1}-v_{(0)}^{n}) \big(v_{(0)}^{n+1} +v_{(0)}^{n} -2v_{R,(0)} \big)\Big),
\nonumber\\
E_2&=\partial_y\Big((v_{(0)}^{n+1}-v_{(0)}^n)(u_{(0)}^{n+1}-u_{R,(0)}) +(u_{(0)}^{n+1}-u_{(0)}^n)(v_{(0)}^n-v_{R,(0)})\Big).
\label{V_Theorem1_eq8}
\end{align}
It is now straightforward to estimate these terms as in (\ref{V_Theorem1_2}). The second momentum equation (\ref{V_zero-order-eq3}) can be treated similarly.

\qed

If we denote $\delta^{n}:= \|\uu^{n}_{(0)}-\uu^{n}_{R,(0)}\|_{W^{1,\infty}}$, the consistency error estimate (\ref{V_Theorem1_2}) is of the order
\begin{equation}
|\mathcal{E}^{n+1}|\leq C\Delta t(\Delta t+\delta^{n}).
\label{V_AP_eq1}
\end{equation}
The \kuc scheme is based on the choice $\uu^{n}_{R,(0)}=\uu^{n}_{(0)}$, hence $\delta^{n}=0$ and the consistency error satisfies
\begin{equation}
\mathcal{E}^{n+1}=\OO(\Delta t^2).
\label{V_AP_eq2}
\end{equation}
On the other hand, for the RS-IMEX scheme, we take $\uu^{n}_{R,(0)}=\uu_{\mathrm{ref}}(t_{n})$, hence $\delta^{n}=\OO(\Delta t)$ and again $\mathcal{E}^{n+1}=\OO(\Delta t^2)$.
We note that in both cases the consistency error is of the second order which is one order higher than the error of approximating the time derivative in (\ref{V_Theorem1_1}). We call this property \emph{superconsistency} of the flux approximation.

We note that this phenomenon might explain the excellent performance of the \kuc scheme for computing steady state solutions, where the time derivative (approximated by a first order difference) is close to zero and the consistency error is of second order due to (\ref{V_AP_eq2}).

\subsection{Well prepared initial data}
Taking into account the results from the previous sections, we do now assume that our initial conditions are \emph{well-prepared}, physically speaking, this means that those initial data do not contain acoustics. Since acoustics are $\OO(\varepsilon)$ perturbations of density, pressure and divergence of velocity, this assumption amounts to having only $\OO(\varepsilon^2)$ perturbations in these quantities.

\begin{definition}
\label{V_def_well_posed}
We say that the initial data are \emph{well prepared} if
\begin{align}
\rho^0= \const+\OO(\varepsilon^2), \qquad p^0=\const+\OO(\varepsilon^2), \qquad \nabla\cdotp \uu^0=\OO(\varepsilon^2).
\end{align}
\end{definition}

We note that if the mentioned quantities possess Hilbert expansions, Definition \ref{V_def_well_posed} amounts to $\rho_{(1)}^0=p_{(1)}^0=\nabla\cdotp\uu_{(1)}^0=0$. Now we prove that if the initial data are well prepared then also $\rho^n={\const}+\OO(\varepsilon^2)$, $p^n={\const}+\OO(\varepsilon^2)$ and $\nabla\cdotp \uu^n=\OO(\varepsilon^2)$ for all $n$.

\begin{theorem}
\label{V_Theorem2} Let the assumptions of Theorem \ref{V_Theorem1} hold. Assume also that the initial data are well prepared in the sense of Definition \ref{V_def_well_posed} and that $\rho_{R,(1)}^n=0$ for all $n$. Then $\rho_{(1)}^n=p_{(1)}^n =\nabla\cdotp\uu_{(1)}^n=0$ for all $n$.
\end{theorem}
\proof
We collect the $\varepsilon^1$ terms of the mass equation from scheme (\ref{V_Unified_scheme}):
\begin{equation}
\frac{\rho^{n+1}_{(1)}-\rho^{n}_{(1)}}{\Delta t} +\nabla\cdotp\big(\rho^{n+1}_{(0)}\uu^{n+1}_{(1)} +\rho^{n+1}_{(1)}\uu^{n+1}_{(0)}\big)=0.
\label{V_first-order-eq1}
\end{equation}
Similarly, we collect the $\varepsilon^1$ terms of the energy equation from scheme (\ref{V_Unified_scheme}), taking into account (\ref{V_rho_Taylor}):
\begin{align}
&\frac{E^{n+1}_{(1)}-E^{n}_{(1)}}{\Delta t} +\nabla\cdotp\bigg( \big(E^n_{(0)}+p^n_{(0)}\big)\uu^n_{(1)} +\big(E^n_{(1)}+p^n_{(1)}\big)\uu^n_{(0)} -\gamma\frac{E_{R,(0)}\uu_{R,(0)}}{\rho_{R,(0)}} (\rho^{n+1}_{(1)}-\rho^{n}_{(1)})\nonumber\\
&\quad-\gamma\Big(\frac{E_{R,(0)}\uu_{R,(1)}+E_{R,(1)}\uu_{R,(0)}}{\rho_{R,(0)}} -\frac{E_{R,(0)}\uu_{R,(0)}(\rho_{R,(1)})^2}{\rho_{R,(0)}}\Big) (\rho^{n+1}_{(0)}-\rho^{n}_{(0)})
\nonumber\\
&\quad+\gamma\frac{E_{R,(0)}}{\rho_{R,(0)}} \big(\rho^{n+1}_{(0)}\uu^{n+1}_{(1)} +\rho^{n+1}_{(1)}\uu^{n+1}_{(0)} -\rho^{n}_{(0)}\uu^{n}_{(1)} -\rho^{n}_{(1)}\uu^{n}_{(0)}\big)
\nonumber\\
&\quad+\gamma\Big(\frac{E_{R,(1)}}{\rho_{R,(0)}} -\frac{E_{R,(0)}\rho_{R,(1)}}{(\rho_{R,(0)})^2}\Big) \big(\rho^{n+1}_{(0)}\uu^{n+1}_{(0)} -\rho^{n}_{(0)}\uu^{n}_{(0)})
\nonumber\\
&\quad+\gamma \uu_{R,(0)}(E^{n+1}_{(1)}-E^{n}_{(1)})+\gamma \uu_{R,(1)}(E^{n+1}_{(0)}-E^{n}_{(0)})\bigg)=0. \label{V_first-order-eq4}
\end{align}
Now we proceed similarly as in the proofs of Lemmas \ref{V_Lemma1a} and \ref{V_Lemma2}. We integrate (\ref{V_first-order-eq4}) over $\Omega$ and apply Green's theorem. Similarly as in (\ref{V_zero-order-eq3a}), the resulting boundary terms are equal to zero due to boundary conditions. This gives us $E^{n+1}_{(1)}=E^{n}_{(1)}$ for all $n$. Consequently also $p^{n+1}_{(1)}=p^{n}_{(1)}$ for all $n$, by taking the $\varepsilon^1$ terms in (\ref{V_EqState}). This implies that $p^{n}_{(1)}=p^{0}_{(1)}=0$ for all $n$.

We proceed by induction and assume that the assumptions of the theorem hold on time level $t_n$. Gathering the assumptions and all previous results, we have that $E^{n}_{(0)}, E^{n}_{(1)}, p^{n}_{(0)}$ and $p^{n}_{(1)}$ are independent of $x$ and $n$, $\nabla\cdotp\uu^{n}_{(0)}= \nabla\cdotp\uu^{n+1}_{(0)} =\nabla\cdotp\uu^{n}_{(1)} =\nabla\cdotp\uu_{R,(0)}=0$ and $\rho^{n+1}_{(0)}=\rho^{n}_{(0)}$. These results allow us to simplify (\ref{V_first-order-eq4}) to
\begin{equation}
-\uu_{R,(0)}\nabla\cdotp(\rho^{n+1}_{(1)}-\rho^{n}_{(1)}) +\nabla\cdotp\big(\rho^{n+1}_{(0)}\uu^{n+1}_{(1)} +\rho^{n+1}_{(1)}\uu^{n+1}_{(0)} \big)=0.
\label{V_first-order-eq4a}
\end{equation}
The second term can be substituted into the mass equation (\ref{V_first-order-eq1}) to obtain
\begin{equation}
\frac{\rho^{n+1}_{(1)}-\rho^{n}_{(1)}}{\Delta t} +\uu_{R,(0)}\nabla\cdotp(\rho^{n+1}_{(1)}-\rho^{n}_{(1)})=0.
\label{V_first-order-eq1a}
\end{equation}
Now we can proceed similarly as in the proof of Lemma \ref{V_Lemma2} -- we multiply (\ref{V_first-order-eq1a}) by $\rho^{n+1}_{(1)}-\rho^{n}_{(1)}$ and apply Green's theorem. All resulting integral terms vanish either due to boundary conditions or since $\nabla\cdotp\uu_{R,(0)}=0$. This implies that $\rho^{n+1}_{(1)}-\rho^{n}_{(1)}=0$, hence, by induction $\rho^{n+1}_{(1)} =\rho^{0}_{(1)}=0$. Using this fact in (\ref{V_first-order-eq1}) implies $\nabla\cdotp\uu_{(1)}^{n+1}=0$. This completes the proof.

\qed

\section{Existence of the Hilbert expansion}
\label{section:Hilbert-expansion}

It is not clear whether the Hilbert expansion at the new time level $n+1$ used in Sec. \ref{section:linearly-implicit-schemes} exists. In most AP proofs this is assumed, and only a few authors, see e.g., \cite{Bispen2017222,BispenDiss} explicitly show it. In this work, we will, for a restricted, yet instructive, case show that this Hilbert expansion exists.
The following assumptions on domain and solutions are used:
\begin{assumption}\label{as:1}
 Assume that boundary conditions are periodic, and that the domain $\Omega \subset \R$. For the sake of simplicity, take $\Omega = [-\pi,\pi]$. (This last assumption is of course not crucial.)
 Assume that all the occurring quantities are sufficiently smooth. More precisely, we assume that the components of $\ww$ are in $H^{\infty}$, with
 \begin{align*}
  H^{\infty} := \left\{\varphi \in L^2(\Omega) \ | \ \sum_{k \in \Z} \left(1+|k|^2\right)^p |\widehat \varphi(k)|^2 < \infty, \  \forall p \in \N\right\}.
 \end{align*}
 $\widehat \varphi(k)$ denote the Fourier coefficients of $\varphi$.
 Note that the severe smoothness condition can be somewhat relaxed.
\end{assumption}

To simplify the analysis, we make the following assumption:
\begin{assumption}\label{as:2}
 Assume that $\refs \ww^n$ is constant in space. (Note that in the sequel, we will omit the superscript $n$ and simply write $\ww$.)
\end{assumption}

\begin{remark}
 It is clear that this is not the most general case; still, it is a very important step towards the full AP analysis.
 %
\end{remark}
Because of the assumptions made above, we can consider the slightly different, yet equivalent formulation of \eqref{V_Unified_scheme}, namely
\begin{align}
 \label{eq:deltaform}
 {\pw^{n+1}} + \dt \partial_x \left(\ffc'(\refs \ww) \pw^{n+1} \right) + \HH^{n} = 0,
\end{align}
where we have defined
\begin{align}
  \pw := \ww - \refs \ww.
\end{align}
$\HH^{n}$ covers all the terms that only depend on time level $n$.
For later reference, we denote
\begin{align*}
 \HH^n =: \left( \delta\rho^*, \delta (\rho \uu)^*, \delta E^*\right)^T.
\end{align*}
The inductive proof of the existence of the Hilbert expansion heavily relies on the fact that 'known' quantities at time level $n$ are assumed to have a Hilbert expansion. Then, also $\HH^n$ has a Hilbert expansion:
\begin{lemma}
 Assume that $\pw^{n}$ possesses a Hilbert expansion. Then the terms collected in $\HH^{n}$ have a Hilbert expansion.
\end{lemma}

In the case we are considering here, i.e., $\Omega \subset \R$, there holds
\begin{equation}
 \label{eq:fdw}
 \ffc'(\refs \ww) \pw^{n+1} = \left( \begin{array}{c}
                                    \delta (\rho \uu) \\
                                    -\refs \uu^2 \delta(\rho) + 2 \refs \uu \delta(\rho \uu) + \frac{p_L}{\eps^2} \\
                                    -\frac{\refs \uu \refs E}{\refs \rho} \delta \rho + \frac{\refs E}{\refs \rho} \delta(\rho \uu) + \refs \uu \delta E - \frac{\refs \uu \refs p}{\refs \rho} \delta \rho + \frac{\refs p}{\refs \rho} \delta (\rho \uu) + \refs \uu p_L
                                  \end{array}
                           \right),
\end{equation}
where we have defined the linearized pressure
\begin{align}
 \label{eq:defpl}
 p_L := (\gamma - 1) \left( \delta E - \frac{\eps^2}{2} \left(-\refs \uu^2 \delta \rho + \refs \uu \delta(\rho \uu) \right) \right).
\end{align}
Note that we have omitted the index $n+1$ on the right-hand side for the sake of a clearer presentation.
\begin{remark}
 It will be crucial for the proof to follow that $p_L = \const + \OO(\eps^2)$. This can already be seen from \eqref{eq:fdw}, because the only term that could destroy a Hilbert expansion is $\frac {p_L}{\eps^2}$. There is a divergence in front, so $p_L$ being constant up to $\eps^2$ is the right choice.
\end{remark}
In the following, we aim to reformulate eq. \eqref{eq:deltaform} in terms of $p_L$. To this end, we first define an operator acting on momentum.
\begin{definition}
 Define the operator $\theta$ through
 \begin{align*}
  \theta: H^{\infty} \rightarrow H^{\infty}, \qquad \mm \mapsto \left(\Id + 2 \dt \refs \uu \partial_x \cdot + \dt^2 \refs \uu^2 \partial_{xx} \cdot \right) \mm.
 \end{align*}

\end{definition}

\begin{lemma}\label{la:theta}
 There holds:
 \begin{enumerate}
  \item $\theta$ is linear.
  \item $\theta$ is invertible.
  \item If a smooth function $\mm^*$ has a Hilbert expansion, then both $\theta(\mm^*)$ and $\theta^{-1}(\mm^*)$ have a Hilbert expansion.
  \item There holds: $\partial_x \theta(\mm) = \theta(\partial_x \mm)$. The same is true for the inverse of $\theta$.
 \end{enumerate}

\end{lemma}
The proof of the lemma is rather straightforward, which is why we omit it here.

Using the operator $\theta$ it is possible to express the momentum at time level $n+1$ as a function of $p_L$. What we are doing here is very similar to the work of \cite{BispenDiss}, in the discrete case, it could be interpreted as a Gaussian elimination procedure.
\begin{lemma}
 There holds:
    \begin{align}
     \label{eq:dru}
     \delta(\rho\uu)^{n+1} = -\frac{\dt}{\eps^2}\theta^{-1}(\partial_x p_L^{n+1}) + \delta(\rho\uu)^{**},
    \end{align}
    with $\delta(\rho\uu)^{**}$ being a quantity that possesses a Hilbert expansion.
\end{lemma}
\begin{proof}
 There holds
 \begin{align}
    \label{eq:dr}
    \delta \rho^{n+1}       &= -\dt \partial_x \delta(\rho \uu)^{n+1} + \delta \rho^*.
 \end{align}
 Plugging this into the momentum equation yields (note that, again, we have omitted the time level $n+1$ on the right-hand side)
 \begin{align*}
 \delta (\rho \uu)^{n+1} &= -\dt \partial_x \left(-\refs \uu^2 \delta(\rho) + 2 \refs \uu \delta(\rho \uu) + \frac{p_L}{\eps^2} \right)  + \delta(\rho \uu)^* \\
                         &= -\dt \partial_x  \left(-\refs \uu^2 \left(-\dt \partial_x \delta(\rho \uu) \right) + 2 \refs \uu \delta(\rho \uu) + \frac{p_L}{\eps^2} \right)  + \delta(\rho \uu)^{+}.
\end{align*}
By $\delta (\rho \uu)^{+}$ we denote terms that are known to have a Hilbert expansion in $\eps$.
Rearranging terms yields
\begin{align*}
 \underbrace{\left(\Id + 2 \dt \refs \uu \partial_x  + \dt^2 \refs \uu^2 \partial_{xx} \right)}_{\theta} \delta(\rho\uu) = -\frac{\dt}{\eps^2} \partial_x p_L + \delta(\rho \uu)^+.
\end{align*}
Exploiting the properties of $\theta$ formulated in Lemma  \ref{la:theta} yields the claim.
\end{proof}

Based on this lemma, we can find that $p_L$ fulfills a third-order differential equation:
\begin{lemma}
 Let $p_L$ be given as in \eqref{eq:defpl}. Then $p_L$ satisfies at time level $n+1$ the equation
 \begin{align}
  \label{eq:pl}
  \omega_0 p_L^{n+1} + \omega_1 \partial_x p_L^{n+1} + \frac{\omega_2}{\eps^2} \partial_{xx} p_L^{n+1} + \frac{\omega_3}{\eps^2} \partial_{xxx}  p_L^{n+1} = p_L^{*},
 \end{align}
 with the constants $\omega_i$ being defined by
    \begin{alignat}{3}
    \label{eq:omega}
    \omega_0 &= \frac{-1}{\gamma-1}, &\quad&
    \omega_2 &=& \frac{\dt^2}{\refs \rho} \left(\frac{-\gamma-5}{\gamma-1}\refs E + \frac{\gamma^2 + 5}{(\gamma-1)^2}\refs p \right) \\
    \omega_1 &= -\dt \refs \uu \frac{5 + \gamma}{2(\gamma-1)}, &\quad&
    \omega_3 &=& \frac{\dt^3 \refs \uu}{\refs \rho} \left(\frac{-2}{\gamma-1} \refs E + \frac{\gamma^2-\gamma+2}{(\gamma-1)^2} \refs p\right);
    \end{alignat}
    and $p_L^* \in H^{\infty}$ being a function that possesses a Hilbert expansion.
\end{lemma}

\begin{proof}
The proof consists of lengthy and tedious, but rather straightforward computations.
The important steps are the following:
\begin{itemize}
 \item First, write $\delta E^{n+1}$ explicitly based on \eqref{eq:fdw}. Use \eqref{eq:dr} and \eqref{eq:dru} to express all quantities $\delta \rho$ and $\delta(\rho \uu)$ in terms of $p_L$. Substitute $E^{n+1}$ on the right-hand side by using the definition of $p_L$ in \eqref{eq:defpl}. Then, apply $\theta$ to the equation, which results in
 \begin{align}
 \label{eq:te1}
 \theta(\delta E^{n+1}) = \omega_0^l p_L + \omega_1^l \partial_{x}  p_L + \omega_2^l \partial_{xx} p_l + \omega_3^l \partial_{xxx} p_L + \delta E^{**}.
\end{align}
As above, $\delta E^{**}$ is a smooth term having a Hilbert expansion. The constants $\omega_i^l$ are given by
\begin{alignat*}{3}
 \omega_0^l &= 0, &\quad&
 \omega_2^l &=& \frac{\dt^2}{\eps^2\refs \rho} \left(\frac{\gamma^2+\gamma+2}{(\gamma-1)^2}\refs p - \frac{2\gamma+2}{\gamma-1}\refs E \right)  \\
 \omega_1^l &= -\frac{\gamma}{\gamma - 1} \dt \refs \uu, &\quad&
 \omega_3^l &=& \frac{\dt^3 \refs \uu}{\eps^2 \refs \rho} \left(\frac{\gamma}{\gamma -1}\refs p - \frac{2}{\gamma -1} \refs E \right).
\end{alignat*}
 \item Second, write $\delta E^{n+1}$ explicitly, this time based on the definition of $p_L$ in \eqref{eq:defpl}, substitute $\delta \rho$ and $\delta (\rho \uu)$ accordingly. Applying $\theta$ on both sides then yields
 \begin{align}
 \label{eq:te2}
 \theta(\delta E^{n+1}) = \omega_0^r p_L + \omega_1^r \partial_{x} p_L + \omega_2^r \partial_{xx} p_l + \omega_3^r \partial_{xxx} p_L + \delta E^{***}.
\end{align}
Again, $\delta E^{***}$ is a smooth term with a Hilbert expansion. The constants $\omega_i^r$ are given by
\begin{alignat*}{3}
 \omega_0^r &= \frac{1}{\gamma-1}, &\quad&
 \omega_2^r &=& \frac{\dt^2}{\eps^2\refs \rho } \left(\frac{3-\gamma}{\gamma-1} \refs E - \frac{3-\gamma}{(\gamma-1)^2} \refs p\right)\\
 \omega_1^r &= \dt\refs \uu \frac{5-\gamma}{2(\gamma-1)}, &\quad&
 \omega_3^r &=& 0.
\end{alignat*}
\item Equating \eqref{eq:te1} and \eqref{eq:te2} and subtracting the constants yields the claim.
\end{itemize}

\end{proof}

\begin{lemma}\label{la:wtwt}
 Let $\gamma \geq 1.$ Then $\omega_2$ and $\omega_3$ cannot be zero simultaneously.
\end{lemma}
\begin{proof}
 Assume that $\omega_2 = 0$ and $\omega_3 = 0$. Then there holds
 \begin{align*}
  \refs E = \frac{\gamma^2+5}{(\gamma-1)(\gamma+5)} \refs p
 \end{align*}
 and
 \begin{align*}
  \refs E = \frac{\gamma^2-\gamma+2}{2(\gamma-1)} \refs p.
 \end{align*}
 Hence,
 \begin{align*}
  \frac{\gamma^2+5}{(\gamma-1)(\gamma+5)} = \frac{\gamma^2-\gamma+2}{2(\gamma-1)}.
 \end{align*}
 The only roots of this equation are $\gamma = -3$ and $\gamma = 0$, they are hence outside the range of $\gamma$.
\end{proof}

\begin{theorem}\label{thm:pl}
 Let $\gamma \geq 1$. Furthermore (as in this whole section), assume that Assumptions \ref{as:1} and \ref{as:2} hold. Then $p_L$ fulfilling the equation \eqref{eq:pl} has a Hilbert expansion, in particular it holds
 \begin{align*}
  p_L = \const + \OO(\eps^2).
 \end{align*}

\end{theorem}
\begin{proof}
 Note that $p_L$ fulfills the equation
 \begin{align*}
  \omega_0 p_L + \omega_1 \partial_{x} p_L + \frac{\omega_2}{\eps^2} \partial_{xx}  p_L + \frac{\omega_3}{\eps^2} \partial_{xxx}  p_L = p_L^*,
 \end{align*}
 see \eqref{eq:pl}; with $p_L^*$ having a Hilbert expansion. Due to Lemma  \ref{la:wtwt} $\omega_2$ and $\omega_3$ cannot be zero simultaneously.
 Because we are operating under periodic boundary conditions, we apply the Fourier expansion for $p_L$
 \begin{align*}
  p_L(x) := \sum_{k \in \Z} \widehat {p_L}(k) e^{ikx}.
 \end{align*}
 Plugging this into \eqref{eq:pl} yields the algebraic equation for $\widehat p_L(k)$
 \begin{align*}
  \frac 1 {\eps^2} \left(\eps^2\omega_0 + \eps^2i k \omega_1 - {\omega_2} k^2 - {\omega_3} i k^3 \right) \widehat {p_L}(k) = \widehat {p_L^*}(k),
 \end{align*}
 where $\widehat {p_L^*}(k)$ denotes the Fourier coefficients of the right-hand side. Because we know that the right-hand side has the Hilbert expansion, we also know that there exists a Hilbert expansion for $ \widehat {p_L^*}(k)$. In particular, with respect to $\eps$, we have $\widehat {p_L^*}(k) = \OO(1)$.
 The Fourier coefficients of $p_L$ are hence given by
 \begin{align*}
  \widehat {p_L}(k) = \frac{\eps^2 \widehat {p_L^*}(k)}{\eps^2\omega_0 + \eps^2i k \omega_1 - {\omega_2} k^2 - {\omega_3} i k^3}.
 \end{align*}
 For $k = 0$ this yields 
 \begin{align*}
  \widehat {p_L}(0) = \frac{\widehat {p_L^*}(0)}{\omega_0} = \OO(1),
 \end{align*}
 while for $k \neq 0$, there holds (note that $\omega_2$ and $\omega_3$ are not zero simultaneously!)
 \begin{align*}
  \widehat {p_L}(k) = -\eps^2 \frac{\widehat {p_L^*}(k)}{\omega_2 k^2 + \omega_3 ik^3} + \OO(\eps^3) = \OO(\eps^2).
 \end{align*}
 Consequently, we have
 \begin{align*}
  p_L(x) = \widehat {p_L}(0) + \sum_{k \in \Z^{\neq 0}} \widehat {p_L}(k) e^{ikx} = \const + \OO(\eps^2),
 \end{align*}
which concludes the proof.

\end{proof}

The following corollary guarantees the existence of a Hilbert expansion having the information on $p_L$. 
\begin{corollary}
 Under the assumptions made in Theorem  \ref{thm:pl}, $\pw^{n+1}$ has a Hilbert expansion, i.e., it can be written as
 \begin{align*}
  \pw^{n+1} = \pw_0^{n+1} + \eps \pw_1^{n+1} + \eps^2 \pw_2^{n+1} + \ldots
 \end{align*}
\end{corollary}
\begin{proof}
 Due to \eqref{eq:dru}, $\delta(\rho \uu)^{n+1}$ can be written as
 \begin{align*}
  \delta(\rho \uu)^{n+1} = -\dt \theta^{-1}\left(\frac{\partial_x p_L^{n+1}}{\eps^2}\right) + \delta(\rho \uu)^{**}.
 \end{align*}
 Because $\frac{\partial_x p_L^{n+1}}{\eps^2} = \OO(1)$ and the properties of $\theta^{-1}$, see Lemma  \ref{la:theta}, also $\delta(\rho\uu)^{n+1}$ can be written in terms of a Hilbert expansion. Due to \eqref{eq:dr} this property carries over to $\delta\rho^{n+1}$.
 Now, as $p_L$, $\delta \rho$ and $\delta (\rho \uu)$ have the Hilbert expansions, it is clear that also $\delta E$ has the Hilbert expansion, too, due to \eqref{eq:defpl}. This proves the claim.
\end{proof}

\section{Conclusion and Outlook}\label{sec:conout}
In this work we have introduced and analysed a class of linearly implicit methods for the discretization of the full Euler equation that unifies several already existing schemes, in particular the \kuc and the RS-IMEX scheme. We have shown that this class of methods is asymptotically consistent and exhibits a phenomenon that we call superconsistency, i.e., the consistency of the flux approximation is higher than expected. Furthermore, for a prototype example, we have shown that this unified class of methods possesses the Hilbert expansion in the case of the full Euler equations which is, to the best of our knowledge, a novel contribution.

Ongoing work focuses on the extension of the analysis, in particular the existence of the Hilbert expansion, to more general situations in multiple dimensions. It is unclear whether the Fourier analysis is then still a suitable framework, as the straightforward extension of the approach we presented here is severely more complicated and it is restricted to the periodic boundary conditions.
Finally, it remains to  investigate numerically the efficiency and accuracy of the proposed splittings in general experiments.

\bibliographystyle{siam}
\bibliography{references}

\end{document}